\documentclass[12pt,a4paper]{article}

\usepackage{
   amsmath,
   amsfonts,
   amssymb,
   graphicx,
   epsfig,
   epstopdf,
   booktabs,
   algorithm,
   algorithmic,
   rotating,
   subcaption,
   pbox
   }

\usepackage[usenames,dvipsnames]{color}
\usepackage[table]{xcolor}
\usepackage{dcolumn}
\newcolumntype{d}{D{.}{.}{3.2}}
\newcolumntype{s}{D{.}{.}{6.2}}
\newcolumntype{z}{D{.}{.}{4.2}}
\newcolumntype{x}{D{.}{.}{2.2}}

\newtheorem{theorem}{Theorem}
\newtheorem{thm}[theorem]{Theorem}

\newtheorem{lemma}[theorem]{Lemma}

\newtheorem{definition}[theorem]{Definition}
        {\hspace*{\fill}$\Box$\par\vspace{4mm}}

\newenvironment{proof}{\noindent{\em Proof.}}%
        {\hspace*{\fill}$\Box$\par\vspace{4mm}}

\newcommand{\qed}{\hfill$\Box$\par\smallskip\noindent}

\title{A new upper bound for the multiple knapsack  problem
}

\author{Paolo Detti \thanks{Dipartimento di
Ingegneria dell'Informazione e Scienze Matematiche, Universit\`a di Siena, Via Roma 56, 53100 Italy, e-mail: detti@dii.unisi.it, tel.: +39 0577-235892, fax: +39 0577-233602}}

\begin{document}
\maketitle
\begin{abstract}
In this paper, a new upper bound for the Multiple Knapsack Problem (MKP) is proposed, based on the idea of relaxing MKP to a  {\em  Bounded Sequential  Multiple Knapsack Problem}, i.e., a multiple knapsack problem in which item sizes are divisible. Such a relaxation, called  sequential relaxation, is obtained by suitably replacing the items of a MKP instance with items with divisible sizes. Experimental results on benchmark instances show that the upper bound is effective when the ratio between the number of items and the number of knapsacks is small. 

\end{abstract}

\section{Introduction}\label{sec:intro}

Given a  set of $n$ items with weigths $w_1,\ldots,w_n$
and profits $p_1,\ldots,p_n$, and a set of $m$ knapsacks  with capacities $c_1,\ldots,c_m$, the 0--1 Multiple Knapsack Problem (MKP) problem consists in  packing items into the knapsacks, in such way that the total weight of the items assigned to a knapsack does not exceed its capacity. The objective is to maximize the total profit of the assigned items.
Let $x_{ij}$ be equal to 1 if item $j$ is assigned to knapsack $i$, an ILP formulation for MKP reads as 
{
\begin{align}
&\max \sum\limits_{i=1}^{m}\sum\limits_{j=1}^{n} p_jx_{ij}\label{ilp:obj}\\
&\sum\limits_{j=1}^{n}w_jx_{ij}  \leq c_i \text{ for } i=1,\ldots,m\label{c1}\\
 &\sum\limits_{i=1}^{m} x_{ij} \leq 1 \text{ for } j=1,\ldots, n  \label{c2}\\
&x_{ij} \in \{0,1\}  \text{ for } i=1,\ldots,m \quad j=1,\ldots, n \label{ilp:bin}
\end{align}
}
The objective function \eqref{ilp:obj} accounts for the maximization of the total profit.
Constraints \eqref{c1}  limit  the total weight of items
assigned to each knapsack. Constraints \eqref{c2}  state that an item can be assigned at most to one knapsack.


In this paper, a new upper bound for MKP is proposed, based on the idea of relaxing the problem to a  {\em  Bounded Sequential  Multiple Knapsack Problem} (BSMKP) \cite{DettiIPL}, i.e., a multiple knapsack problem in which item sizes are divisible. Such a relaxation, called {\em sequential relaxation} in what follows, is obtained by suitably replacing the items of an MKP instance with items with divisible sizes. 
In BSMKP, multiple copies may exist of each item. Hence, items can be partitioned into classes, each class containing items with the same profit and weight. BSMKP can be polynomially solved in $O(\bar n^2 + \bar nm)$ time \cite{DettiIPL}, where $\bar n$ is the number of item classes (the complexity reduces to $O(\bar n\log \bar n + \bar nm)$ when a single copy of each item exists).
We prove that the  upper bound provided by the sequential relaxation  is always not  worse than the linear relaxation of model  \eqref{ilp:obj} --\eqref{ilp:bin}. Computational results on benchmark instances from the literature show that, in comparison with a classical upper bound for MKP  \cite{MT1981, Pis99}, the sequential upper bound is particularly effective when the ratio $n/m$ is  small, i.e., $n/m\le 3$. 

The paper is organized as follows. Section \ref{sec:lit} reports results from the literature. 
In Section \ref{sec:sequb}, the sequential relaxation is formally defined and described. In Section \ref{sec:transf}, a procedure for generating a series of different sequential relaxations is proposed. In Section \ref{sec:res},  computational experiments on benchmark instances are presented. Finally, conclusions follow.

\section{Literature results}\label{sec:lit}
 
MKP is a strongly NP-hard problem intensively studied in the literature. For reviews on MKP and its variants, we refer the reader to the books by Martello and Toth \cite{MT1990} and Kellerer {\em et al.} \cite{Kellerer}.   
Effective exact algorithms for MKP include the bound and bound method proposed in \cite{MT1981}, called MTM, turned out to be computationally much faster than the previous approaches proposed in the literature.
Pisinger in  \cite{Pis99} derived from MTM a more efficient exact procedure, called MULKNAP, capable of solving to optimality large-size instances with up to 100,000 items and 5 or 10 knapsacks. However, none of the algorithms were able to solve instances with small values of $n/m$.
MTM and  MULKNAP employ upper bound computations obtained through the {\em surrogate relaxation} \cite{MT1981} of the capacity constraints.

Recent contributions to MKP have been presented by Chekuri and Khanna \cite{Chekuri}, Fukunaga and Korf  \cite{Fukunaga07}, Fukunaga  \cite{Fukunaga11}, Jansen \cite{Jansen}, Lalami {\em et al.}  \cite{Lalami}, and Balbal {\em et al.}  \cite{Balbal}. In   \cite{Fukunaga07},  a branch-and-bound algorithm is proposed, based on a bin-oriented branching structure and a  dominance criterion. The algorithm turned out to be effective  for relatively small $n/m$ ratios (i.e., $n/m \sim 4$). More recently, Fukunaga \cite{Fukunaga11}  proposes a solution approach for
MKP, (extending the one proposed in \cite{Fukunaga07}, based on the integration of path-symmetry and path-dominance criteria and bound-and-bound techniques \cite{MT1981, Pis99}. The solver  appears to be effective on instances with high $n/m$ ratios. Dell'Amico {\em et al.} \cite{Dell'Amico} propose two new pseudo-polynomial formulations, and  an exact effective method based on the hybrid combination of several techniques, called Hy-MKP. On benchmark instances, Hy-MKP attains very good performances, failing  on some instances with  $ n/m = 3$ and on few instances with ratios 4 and 5.

As the computational experiments show, the sequential relaxation  proposed in this paper, based on solving a BSMKP problem, turns out to be effective when $n/m$ is small, i.e.,   $n/m\le 3$. 
 
 BSMKP has been addressed in the literature in \cite{DettiIPL, DettiIPoly}.  The single knapsack version of BSMKP is known in the literature as  {\em
sequential knapsack problem} (SKP). For the unbounded SKP (i.e., the problem in which an infinite number of copies exists for each item), Marcotte proposes a linear time algorithm \cite{Marcotte}, and Pochet and Wolsey \cite{WolPoc} provide  an explicit polytope description.
 For the bounded SKP, Verhaegh and Aarts present an
$O(n^2\log n)$ algorithm  \cite{ver_aarts},  Hartmann and
Olmstead \cite{HartOlm} propose an $O(n\log n +
\sum\limits_{j=1}^{n}\log b_j)$ algorithm,  where $b_j$ is the number of copies of item $j$, 
and Pochet and  Weismantel \cite{Pochet98onthe} provide a polytope description.\\
 In \cite{DettiIPL}, Detti proposed a polynomial  $O( n^2 +  nm)$ algorithm  for BSMKP. The complexity of the algorithm reduces to $O(n\log n + nm)$ when  a single copy exists for each item. Hence, for SKP, the algorithm presented in \cite{DettiIPL} requires $O(n^2)$ steps for the bounded case  and $O(n\log n)$ steps, when a single copy exists of each item (the same complexity of the algorithm proposed
in \cite{HartOlm}). A complete description of the BSMKP polytope  is presented  in \cite{DettiIPoly}.


\section{The Sequential relaxation}\label{sec:sequb}

The new proposed upper bound is based on relaxing MKP to a Bounded Sequential  Multiple Knapsack Problem. 
 BSMKP can be formally stated as follows. There are a set of items partitioned into
$\bar n$ different classes and a set of $\bar m$ knapsacks.
Each item of class $t$ has a size $s_t \in \mathcal{Z}^+$, a profit
$v_t \in \mathcal{Z}^+$ and an upper bound $b_t \in
\mathcal{Z}^+$, for $t=1,\ldots, \bar n$. Item sizes are divisible, i.e., $s_{t+1}$ is divisible by $s_t$, for $t=1,\ldots, \bar n-1$. 
Each knapsack $i$ has a capacity $\bar c_i \in
\mathcal{Z}^+$. The problem is to find
the number $y_{it}$ of items of class $t$, to be assigned to each knapsack $i$, in such a way that the total profit is maximized. A formulation of BSMKP reads as follows:
{
\begin{align}
&\max \sum\limits_{i=1}^{\bar m}\sum\limits_{t=1}^{\bar n}v_ty_{it}\label{seq:obj}\\
&\sum\limits_{t=1}^{\bar n}s_ty_{it}\leq \bar c_i \text{ for } i=1,\ldots,m\label{seq:c1}\\
 &\sum\limits_{i=1}^{\bar m}y_{it}\leq b_t\text{ for } t=1,\ldots, \bar n  \label{seq:c2}\\
&y_{i,t} \in \mathcal{Z}^+ \label{seq:int}
\end{align}
}
The objective function  \eqref{seq:obj} accounts for the maximization of the total profit. Constraints  \eqref{seq:c1} state that the total size of items
assigned to a knapsack does not exceed its capacity.
Constraints  \eqref{seq:c2} impose that the total number of the assigned  items of class $t$ does not exceed  the upper bound.

\medskip\medskip
Given an instance of MKP, we call {\em sequential relaxation}  the transformation of the MKP instance into an instance of BSMKP.  The sequential relaxation is formally defined in the following.
\begin{definition}
The {\em sequential relaxation}  of an instance $I$ of MKP is an instance   $I_s$ of BSMKP obtained by replacing each item $j$ in $I$ with a set of items $S_j=\{j_{1},\ldots, j_{K_j}\}$ with positive sizes $s_{j_1}\le s_{j_2}\le \ldots \le s_{j_{K_j}}$ and profits $v_{j_1}, v_{j_2} \ldots  v_{j_{K_j}}$, respectively, such that:
\begin{align}
& \sum\limits_{k \in S_j} s_{j_k}= w_j \text{ for } k=1,\ldots,K_j  \label{eq:seq}\\
&  v_{j_k}= {p_j \over w_j} s_{j_k} \text{ for }j =1,\ldots,n, k=1, \ldots,K_j \label{eq:seq1}\\
 &   \text{ the sizes } s_{j_k}  \text{ are divisible, } \text{ for } j =1,\ldots,n, k=1,\ldots,K_j-1.\label{eq:seq2}
\end{align}
All items in $I_s$ with the same size and profit belong to the same class.
The knapsacks in $I_s$ are the same of $I$, i.e.,  $\bar m=m$ and $\bar c_i=c_i$, for $i=1,\ldots,m$.
\end{definition}
Note that, Conditions \eqref{eq:seq} and \eqref{eq:seq1} imply that 
$$ \sum\limits_{k \in S_j} v_{j_k}= p_j  \text{ for } k=1,\ldots,K_j.$$
As shown in \cite{DettiIPL}, $I_s$  can be optimally solved in  $O(\bar n^2 + \bar n m)$ time. 
Many different sequential relaxations  may exist, depending on the set of divisible sizes $s_{j_k}$,  for $j =1,\ldots,n$ and $k=1,\ldots,K_j$, used for generating $I_s$. In fact, if weights $w_j$ and capacities $c_i$ are integers, the sequential relaxation is equivalent to the linear relaxation of model model \eqref{ilp:obj}--\eqref{ilp:bin} when $s_{j_k}=1$,  for $j =1,\ldots,n$ and $k=1,\ldots,K_j$ is set (i.e., each item in $I$ is split into smaller items of size 1).

As an example, let $I$ be a MKP instance with  two knapsacks of capacities $c_1=47$ and $c_2=64$, and  five items, with weights and profits reported in  Columns 2--3 of Table \ref{table1}.
 
 Let $I_s$ be an instance of BSMKP produced by partitioning each item of $I$ into items of divisible sizes $s_1=1$, $s_2=3$ and $s_3=33$. 
 Hence, since $w_1=s_3$, item 1 is not partitioned and is included in $I_s$. Item 2 can be partitioned into two items of size 1 and profit 2, and one item of size 33 and profit 66. From item 3 the following items are generated:  one item of size 1 and profit 2, one item of size 3 and profit 6, and one item of size 33 and profit 66.  
 Item 4 is partitioned into two items of size 1 and profit 1, four items of size 3 and profit 3, and one item of size 33 and profit 33. Finally,  item 5 can be partitioned into one item of size 1 and profit 1, ten items of size 3 and profit 3, and one item of size 33 and profit 33.  
 Columns 4--9 of Table \ref{table1} report all the items of $I_s$. More precisely,  for each item of MKP, Columns 4, 6 and 8  respectively report the number of items  of size $s_1$, $s_2$ and $s_3$ generated in $I_s$, denoted as $\#s_1$, $\#s_2$ and $\#s_3$. Columns 5, 7 and 9 show the  profits, $v_1$, $v_2 $ and $v_3$, of the generated items of size $s_1$, $s_2$ and $s_3$, respectively.  
 Note that,  $\#s_j=b_j$ for $j=1,2,3$. 
 
\medskip\medskip
\begin{table}
\centering
\scriptsize
\begin{tabular}{|c|c|c||c|c||c|c||c|c|}
\hline j & $w_j$& $p_j$&   \#$s_1$ & $v_1$&  \#$s_2$ & $v_2$&  \#$s_3$ & $v_3$\\
\hline  1 & 33&  99        & -&- & -& -        & 1&99\\
\hline  2 & 35& 70          & 2&2 & -& -        & 1&66 \\
\hline  3 & 37& 74          & 1&2 & 1& 6         & 1&66\\
\hline  4 & 47& 47          & 2&1 & 4& 3         & 1&33 \\
\hline  5 & 64& 64         & 1&1 & 10&3         & 1&33 \\
\hline
\end{tabular}
\caption{Instances $I$ of MKP and $I_s$ of BSMKP.} \label{table1}
\end{table}

Let $z_{MKP}$ and $z_{seq}$ be the  optimal solution values of a MKP instance and of any sequential relaxation, respectively. The following lemma holds.

\begin{lemma}\label{lemma:ub}
Given an instance $I$ of MKP, the optimal solution value  $z_{seq}$ of any sequential relaxation $I_s$ is an upper bound to the  optimal solution value  $z_{MKP}$ of $I$, i.e., 
$$z_{MKP}\leq z_{seq}.$$
\end{lemma}
\begin{proof}
Given a feasible solution, $x$, for $I$, let $y$ be a  solution for  $I_s$ obtained by $x$ replacing each item $j$ in $I$ with items  in $S_j$. By Equation \eqref{eq:seq}, $y$ is feasible for  $I_s$, and by Equation \eqref{eq:seq1}, $x$ and $y$  have the same objective function value. This holds even if $x$ is optimal and the thesis follows. 
\end{proof}

Let $z_{LPMKP}$ be the optimal solution of the linear relaxation of the model \eqref{ilp:obj}--\eqref{ilp:bin} for an instance $I$ of MKP. Then, Lemma \ref{lemma:lp} holds.
\begin{lemma}\label{lemma:lp}
Given an instance $I$ of MKP, the upper bound $z_{seq}$ derived from instance   $I_s$ obtained by any sequential relaxation
of $I$ is not bigger than the upper bound  $z_{LPMKP}$ obtained by the linear relaxation of  model \eqref{ilp:obj}--\eqref{ilp:bin}, i.e., 
$$z_{LPMKP}\geq z_{seq}.$$
\end{lemma}
\begin{proof} The Lemma is proved by showing that any feasible solution of $I_s$ is also feasible for the linear relaxation of  model \eqref{ilp:obj}--\eqref{ilp:bin}.
Given a feasible solution, $y$, for $I_s$, let $q_{itj}$ be the number of items of class $t$ (i.e.,  of size $s_{t}$ and profit $v_{t}$) belonging to $S_j$ and assigned to knapsack $i$ in $y$.   
Then, we have $y_{it}=\sum\limits_{j=1}^n q_{itj}$. Note that, by \eqref{eq:seq} and by the definition of $q_{itj}$, we have $\sum\limits_{i=1}^m\sum\limits_{t =1 }^{\bar n} s_{t} q_{itj}\leq w_j$. Let $x$ be  the solution of the linear relaxation of  model \eqref{ilp:obj}--\eqref{ilp:bin} in which
$x_{ij}=\sum\limits_{t =1 }^{\bar n}  s_{t} q_{itj}/ w_j$ is set,  for $j=1,\ldots,n$ and $i=1,\ldots,m$. Then,  $\sum\limits_{i=1}^m x_{ij}=\sum\limits_{i=1}^m\sum\limits_{t =1 }^{\bar n}  s_{t} q_{itj}/ w_j\le 1$ (i.e., $x$ satisfies Constraints  \eqref{c2}). Furthermore, by  \eqref{seq:c1} we have
$$\sum\limits_{j=1}^{n}w_jx_{ij} = \sum\limits_{j=1}^{n}w_j\sum\limits_{t =1 }^{\bar n}  s_{t} q_{itj}/ w_j = 
\sum\limits_{t =1 }^{\bar n}   s_t y_{it}
\leq \bar c_i = c_i,$$
i.e., $x$ satisfies Constraints  \eqref{c1}. Hence, $x$ is a feasible solution for the linear relaxation of  model \eqref{ilp:obj}--\eqref{ilp:bin}.
By formulas \eqref{ilp:obj} and \eqref{seq:obj} and by conditions \eqref{eq:seq1}, the objective functions \eqref{ilp:obj} and \eqref{seq:obj} (computed in $x$ and in $y$, respectively)  have the same objective function value. This holds even if $x$ is optimal and the thesis follows. 
\end{proof}

Hence, the sequential relaxation provides an upper bound that is not worse than the one provided by the linear relaxation of model \eqref{ilp:obj}--\eqref{ilp:bin}.

\medskip
The  exact algorithms  MTM and  MULKNAP \cite{MT1981,Pis99} employ upper bounds computed by the {\em surrogate  relaxation} of the capacity Constraints \eqref{c1}, obtained by replacing them by the single knapsack constraint
$\sum\limits_{i=1}^m\sum\limits_{j=1}^n \pi_i w_j x_{ij} \le c_i,$ where $\pi_1, \ldots, \pi_m$ are non-negative multipliers.
Martello and Toth \cite{MT1981} proved that for any instance of MKP, the optimal choice of multipliers is $\pi_i=k$ for all $i$, where $k$ is a positive constant. Hence, the surrogate relaxation can be found  by solving an ordinary 0--1 Knapsack Problem.
In the following, we denote the upper bound obtained by the surrogate relaxation as $z_{surr}$.

As also shown by the computational analysis in Section \ref{sec:res},  it is not possible to establish a theoretical relationship between the bounds provide by sequential and surrogate relaxations. 

In the following, as an example, we derive  upper bounds from the sequential, surrogate and linear relaxations  for the MKP instance reported in Table \ref{table1}. As shown below, we have that  $z_{seq}$ is smaller than  $z_{LPMKP}$ and $z_{surr}$.
In Pisinger \cite{Pis99}, a procedure is used to tighten the capacity constraints of a MKP instance $I$. In practice, $m$ Subset-sum problems are solved, one for each knapsack, for detecting the maximum knapsack capacities that can be filled by the items in $I$. Then, knapsack capacities are reduced to such maximum values. Note that, for the MKP instance of Table \ref{table1} capacity constraints can not be tightened by this approach.  The optimal solution of the instance consists in assigning items 1 and 3. Hence, $z_{MKP}=173$.  The surrogate bound is found by solving  a single knapsack problem with knapsack capacity $c_1+c_2= 111$. It is easy to see that $z_{surr}=243$, obtained by assigning items  1, 2 and 3 with a total weight of 105.  
 The optimal solution of the sequential relaxation (reported in Columns 4--11 of Table \ref{table1}), is obtained by assigning the new items as follows.  Items assigned to Knapsack 1: the item of size 33 and profit 99,
 the item of size 3 and profit 6, three items of size 1 and profit 2, two items of size 3 and profit 4, two items of size 1 and profit 1.
Items assigned to Knapsack 2: 1 item of size 33 and profit 66, ten items of size 3 and profit 3, 1 item of size 1 and profit 1.  Hence, $z_{seq}=119+97=216$. Finally, the optimal solution value of the linear relaxation of  model \eqref{ilp:obj}--\eqref{ilp:bin} is $z_{LPMKP}=249$.

\section{Generating and solving sequential relaxations}\label{sec:transf} 
In this section, we address the problem of generating  different sequential relaxations from a MKP instance $I$, in order to get small $z_{seq}$ values. 
 In fact, many possible sequential relaxations may exist, depending on the divisible sizes $s_t$ used to generate items in $I_s$. In what follows, we denote {\em  sequential sequence}  the set $S$ containing the divisible sizes of the items of  $I_s$. 
The  procedure described below can be used for finding a series of different  sequential sequences.


Let $j$ be  an item in $I$, let $Q$ be an integer smaller than $w_j$, and let $q\ge 2$ be the biggest  integer smaller than or equal to $Q$ such that $(w_j \mod q)$ is minimum.  
In $I_s$, we denote as {\em reference size}  the value $\bar s=w_j - (w_j \mod q)$. 
In other words, $q$ is the biggest integer not bigger than $Q$ such that $\bar s$ is ``closest'' to $w_j$. 

Note that,  $\bar s$ and $q$ are divisible. At the beginning, $S=\{1, \bar s\}$ is set. Then, the procedure sequentially scans the numbers in the ordered set $T=\{\bar s -q, \bar s -2q, \ldots, q, q-1, q-2,  \ldots, 2\}$:  Whenever an element of the set $T$ is detected  that divides all values in $S$, it is  included  in $S$ (and the search continues considering the next elements of $T$). The procedure ends either when $S$ contains $l_{max}$  elements or when all elements of $T$ are scanned.  

Note that, the above procedure is not polynomial in the input size of the instance $I$, since it depends on the weight $w_j$ and $Q$. However, it can be made faster by suitably choosing small values for $Q$  and $l_{max}$.

Let $S=\{1, s_1, s_2,\ldots, \bar s=s_{\bar n}\}$ be the sequential sequence obtained so far, with $1\le s_1\le s_2\le \ldots s_{\bar n}$.  The items in $I_s$ are generated as follows. Firstly the biggest size $s_{\bar n}$ in $S$ is considered, and, for each item $j$ of $I$, $\lfloor w_j /s_{\bar n} \rfloor$ items are  generated with profit  
$(p_j / w_j) s_{\bar n}$ and size $s_{\bar n}$ and included  in $I_s$. Then, $w_j=w_j- \lfloor w_j /s_{\bar n} \rfloor s_{\bar n}$ is set for all items $j$ of $I$, and the above argument is applied by considering the second biggest size $s_{\bar n-1}$. And so on, until the last size 1 is considered. 
Hence, the instance $I_s$ generated so far will contain at most $n \times \bar n$ item classes.

The overall procedure can be executed more than one time,  by selecting at the beginning a different item $j$ in $I$ (possibly leading to new $q$  and $\bar s$ values). In this way, we get different instances $I_s$, each of them solvable in $O(\bar n^2 + \bar nm)$ by the algorithm proposed in \cite{DettiIPL}. At the end,  the smallest $z_{seq}$ value is returned.  
The overall algorithm is reported in Figure \ref{fig:algo}. In the algorithm,  $Q_{max}$ is an input parameter  used to limit $Q$.  

As an example, let us consider the MKP instance of Columns 1--3 of Table \ref{table1} and let $Q_{max}=10$ and $l_{max}=5$. Let us suppose that  the item $j=5$ is selected. Then, $Q=\min\{Q_{max},w_5\}=\min\{10, 64\}=10$ and $q=8$. Hence, $\bar s=w_5-(w_5 \mod q)=64$, and $S=\{1,64\}$ is initially set. Then, the algorithm scans the ordered sequence  $T=\{\bar s -q=56, \bar s -2q=48, \bar s -3q=40,\ldots, q=8,7,6,\ldots,2\}$ and includes in  $S$ all numbers dividing all the elements of the current set $S$ (until $|S|\le l_{max}$). Hence,   the sequential sequence $S=\{\bar s=s_{\bar n}, \ldots, s_1\}=\{64,32,16,8,1\}$ is get. By solving the BSMKP instance generated from $S$ we get $z_{seq}=249$, equal to the optimal solution of the linear relaxation of  model \eqref{ilp:obj}--\eqref{ilp:bin}. On the the other hand, let us suppose that the algorithm selects item $j=1$. Then, $Q=\min\{Q_{max},w_1\}=\min\{10, 33\}=10$ and $q=3$. Hence, $\bar s=33$, and  $S=\{1,33\}$ is set. Then, the algorithm scans the ordered sequence  $T=\{\bar s -q=30, \bar s -2q=27, \bar s -3q=24,\ldots,6, q=3,2\}$, and produces  the final sequential sequence $S=\{33,3,1\}$. From $S$, we get  the BSMKP instance of Table \ref{table1} with optimal solution value $z_{seq}=216$.

\begin{algorithm}
\caption{Algorithm for generating and solving sequential relaxations of MKP.}\label{fig:algo} 


{\scriptsize
\begin{tabbing}
{\bf Algorithm Sequential Relaxations}\\
{\bf Input:} An instance $I$ of MKP, an integer $Q_{max}$, a maximum number of elements $l_{max}$, a maximum iteration number $It_{max}\le n$;\\
{\bf Output:}  The best upper bound $z_{seq}$;\\
$h=0$; mark all items in $I$ as not visited.\\
{\bf while} ($h\le It_{max}$)\\
{\bf begin }\\
\quad Select a not already visited item $j$ in $I$, and set $Q=\min\{Q_{max}, w_j\}$.\\
\quad Let $q\le Q$ be the biggest integer  such that $(w_j \mod q)$ is minimum. Set $\bar s=w_j-(w_j \mod q)$,  $S=\{\bar s,1\}$.\\
\quad Sequentially scan the ordered set $T=\{\bar s -q, \bar s -2q, \ldots,q, q-1, q-2, \ldots,2\}$ and include in $S$\\
\quad all numbers of $T$ dividing all elements in $S$ until $|S|\le l_{max}$.\\
\quad {\bf while} ($S\ne \emptyset$)\\
\quad {\bf begin }\\
\quad \quad  Let $s_{\bar n}$ be the biggest element in $S$.\\ 
\quad \quad  For all items $l$ in $I$, generate $\lfloor w_l /s_{\bar n} \rfloor$ items   with profit  
$(p_l/ w_l) s_{\bar n}$ and size $s_{\bar n}$.\\
\quad \quad Set $S=S\setminus \{s_{\bar n}\}$.\\
\quad {\bf end }\\
\quad $h=h+1$; mark $j$ as visited.\\
\quad Get $z_{seq}$ by solving $I_s$ by the algorithm proposed in \cite{DettiIPL}.\\
{\bf end}\\
Return the smallest $z_{seq}$ obtained so far.

\end{tabbing}
}
\end{algorithm}


\section{Computational results}\label{sec:res}
In this section, computational results are presented on different sets of benchmark instances, in order to compare  the sequential and the surrogate relaxations, and the linear relaxation of  model \eqref{ilp:obj}--\eqref{ilp:bin}. More precisely, six sets of instances have been used:  the first five sets  are from the literature, while the sixth set contains new instances  with $n/m$ ratios smaller than 2. All the instances were obtained through Pisinger’s instance generator. In all the instances, knapsack capacities have been tightened as proposed in  \cite{Pis99} (by solving a series of Subset-sum Problems), and the surrogate upper bound is computed by the $C$ code developed by Pisinger  \cite{Pis99}. The instance generator  and the code to solve the surrogate relaxation are available  at $
http://hjemmesider.diku.dk/\sim pisinger/codes.html$. The Algorithm \ref{fig:algo} for computing the sequential upper bound  has been  also coded in $C$. In the algorithm, parameters have been set as follows: $Q_{max}=10$, $l_{max}=5$ and $It_{max}=10$. 


All the experiments have been performed on a machine equipped with Intel i7, 2.5 GHz Quad-core processor and 16 Gb of RAM. Gurobi solver has been used to compute the linear relaxations of  the Integer Linear Programming formulation \eqref{ilp:obj}--\eqref{ilp:bin}.

The first five sets of instances contain instances  generated  in Dell'amico {\em et al.} \cite{Dell'Amico}, and firstly proposed by Kataoka and Yamada \cite{Kataoka}  and Fukunaga \cite{Fukunaga11}, denoted as {\em SMALL}, $Fk_1$, $Fk_2$, $Fk_3$ and $Fk_4$ (available at $http://or.dei.unibo.it/library$). The sixth set, denoted as $Set_6$, contains large randomly generated instances with  $n/m < 2$ and is available at $https://www3.diism.unisi.it/\sim detti/SequentialBound.html$. As in  Pisinger  \cite{Pis99}, four classes of correlation are considered: uncorrelated, weakly correlated, strongly correlated, subset-sum.
In the following, the instances are described into detail. 

{\em SMALL} is a set of 180  instances  proposed by Kataoka and Yamada \cite{Kataoka} 
for a variant of the MKP with assignment restrictions, and adapted to MKP by Dell'Amico {\em et al.} \cite{Dell'Amico}  by simply disregarding the additional constraints.  This set contains uncorrelated, weakly correlated and strongly correlated instances with $m \in \{10, 20\}$ and $n \in \{20, 40, 60\}$, for a total of 18 settings (10 instances exist for each setting).  Weights $w_j$ are  uniformly distributed in $[1, 1000]$ in all the {\em SMALL} instances.  In uncorrelated instances, profits $p_j$ are  uniformly distributed in $[1, 1000]$. In weakly correlated  instances, profits  are set as $p_j = 0.6 w_j + \theta_j$, with $\theta_j$ uniformly random in $[1,400]$. In strongly correlated instances, $p_j = w_j +200$ is set.
The knapsack capacities were generated as $c_i =\lfloor{\sigma\lambda_i \sum_{j=1}^n w_j}\rfloor$, with $\lambda$ uniformly distributed in $[0,1]$ such that $\sum_{i=1}^m \lambda_i=1$, and $\sigma \in \{0.25, 0.5, 0.75\}$.
The values of $n$ and $m$ and the correlation classes of the instances in this set are reported in the second row of Table \ref{tab:inst}.

The other set of instances,  i.e., $Fk_1$--$Fk_4$ and  $Set_6$, have been generated as in \cite{Pis99}. 
In $Set_6$, data are generated according to different ranges $R=100, 1000, 10000$, while $R=1000$ has been used in all the instances of Sets  $Fk_1$--$Fk_4$. In 
uncorrelated instances: $p_j$ and $w_j$ are randomly distributed in $[10,R]$. 
In weakly correlated instances, $w_j$ is randomly distributed in $[10, R]$ and $p_j$ is randomly distributed in $[w_j-R/10,w_j+ R/10] $ such that $p_j\ge1$.  In strongly correlated instances, $w_j$ is randomly distributed in $[10,R]$ and $p_j$ is set to $w_j+10$. In subset-sum  instances, $w_j$ is randomly distributed in $[10, R]$ and $p_j$ equals $w_j$.
The  first $m-1$ knapsack capacities $c_i$ are randomly distributed in 
$\left[ 0.4 \sum\limits_{j=1}^{n}w_j/m, 0.6 \sum\limits_{j=1}^{n}w_j/m\right] $
and 
$c_m=0.5 \sum\limits_{j=1}^{n}w_j - \sum\limits_{i=1}^{m-1}c_i. $

The sets $Fk_1$, $Fk_2$, $Fk_3$ and $Fk_4$ contain 480 instances each. They have been generated in \cite{Dell'Amico} and reproduce those used in  \cite{Fukunaga11}. The values of $n$ and $m$ and the correlation classes of the instances in sets $Fk_1$--$Fk_4$ are reported in Rows 3--5 of Table \ref{tab:inst}. Twenty instances exist for each setting. 
$Set_6$ contains 1620 large instances ($n/m$  ranges from 150/80 to 45000/30000), weights and profits are generated for all values of $R=100, 1000, 10000$ as in \cite{Pis99}. The values of $n$ and $m$ and the correlation classes of the instances in set $Set_6$ are reported in Rows 6--7 of Table \ref{tab:inst}. (For each setting, 20 instances exist).

Tables \ref{tab:small}--\ref{tab:Set6-2} report the computational results on the six sets. In the tables, $z_{seq}$,  $z_{surr}$ and $z_{LP}$ are the average values of the sequential, surrogate and linear relaxations, respectively, and $t_{seq}$, $t_{surr}$ and $t_{LP}$ are the related average computational times.
Tables \ref{tab:small}--\ref{tab:FK4} report the results for the sets SMALL and $Fk_1$--$Fk_4$, respectively. 
 In Tables  \ref{tab:small}--\ref{tab:FK4}, ‘‘opt'' is  the optimal  average solution values on each setting, kindly provided by the authors of \cite{Dell'Amico}. A ‘‘-'' in this column means that the optimum is not known for at least one instance of the setting. In Columns 12--14 of the tables, $gap_{se}$, $gap_{su}$ and $gap_{LP}$ are the percentage optimal gaps  of  sequential, surrogate and linear relaxations computed as $(z_{seq} - opt)/opt\times 100$, $(z_{surr} - opt)/opt\times 100$ and $(z_{LP} - opt)/opt\times 100$, respectively.

Table \ref{tab:small} reports the results on the {\em SMALL}  set. The results in each row of the table are average values on 10 instances. The last row of the table reports the average results over all the instances. Observe that, in general, all gaps are big for instances with small ratios $n/m$ and  decrease as the ratios increase. In fact, when $n/m=1$, we have $gap_{se}$, $gap_{su}$ and $gap_{LP}$   equal to about 45\%, 74\% and 87\% on average, respectively, with the sequential bound attaining the best performance (especially on uncorrelated and weakly instances). On instances with $n/m=2$, the sequential and surrogate relaxations produce the best results, with $gap_{se}=5.48$, $gap_{su}=5.63$  and $gap_{LP}=7.66$ on average. In instances with bigger ratios (i.e., $n/m=3,4,6$), $gap_{su}$ is always smaller than $gap_{se}$ (and obviously than $gap_{LP}$). In fact, we have $gap_{su}=0.15$, $gap_{se}=0.59$ and $gap_{se}=0.65$ on average.

The computational results on $FK_1$--$FK_4$ instances are shown  in Tables \ref{tab:FK1}--\ref{tab:FK4}, respectively, where each row report average values on the 20 instances of each setting (with the same $n$, $m$ and correlation class). 
The trends  on these instances are similar to those highlighted on SMALL instances.
In fact, 
the biggest optimality gaps are attained on instances with the smallest ratio $n/m=2$. On these instances, $gap_{se}$ is definitely smaller than $gap_{su}$ and $gap_{LP}$ on uncorrelated and weakly correlated instances, and slightly smaller or equal on strongly correlated instances. When $n/m=3$,  $z_{seq}$ is smaller  than $z_{surr}$ on uncorrelated and weakly correlated instances,  while $z_{surr}$ is smaller on  strongly correlated instances. On instances with biggest ratios, $z_{surr}$ is generally smaller for uncorrelated, weakly and  strongly correlated instances. In all subset-sum instances of sets $FK_1$--$FK_4$, the sequential and surrogate relaxations produce the same bounds, that are equal or very close to the bounds provided by the linear relaxation.

As Tables \ref{tab:FK2}--\ref{tab:FK4} show,  the optimality gaps can not be computed for some instances of sets $FK_2$, $FK_3$ and $FK_4$ with ratios $n/m=3,4,5$, since, a the the best of our knowledge,  the optimal solutions   are  not available for some of these instances  in the literature.  In fact, as shown in the detailed analysis reported in Tables 5 and 6 of \cite{Dell'Amico}, it turns out that the  effective Hy-MKP approach (proposed in \cite{Dell'Amico})  fails to find optimal solutions especially on instances with  $ n/m = 3$, and on some instances with ratios 4 and 5.

For a clearer comparison,   Table \ref{tab:compare} reports the gaps between sequential and surrogate bounds  for the  instances with ratios $n/m=3,4,5$ belonging to the sets $FK_2$, $FK_3$ and $FK_4$. In the table, $g_{se-LP}$ is the percentage gap between $z_{seq}$ and  $z_{LP}$, and $g_{su-LP}$ is the percentage gap between $z_{surr}$ and  $z_{LP}$, computed as 
  $(z_{LP} - z_{seq})/z_{seq}\times 100$ and $(z_{LP} - z_{surr})/z_{surr}\times 100$, respectively. 
  Hence, the bigger  $g_{se-LP}$ and $g_{su-LP}$ are the better  the sequential and surrogate bounds  are.
As shown in Table \ref{tab:compare}, $z_{seq}$ is better than $z_{surr}$ on uncorrelated and weakly correlated instances with ratio 3. In fact, $g_{se-LP}$ and $g_{su-LP}$ respectively are equal to 0.53 and 0.06 on average. On the remaining instances, $z_{seq}$  is essentially equal to $z_{LP}$ while $z_{surr}$ is slightly better with  $g_{su-LP}=0.02$ on average.

The computational times on SMALL and $FK_1$--$FK_4$ instances are very small for  the sequential (1 millisecond or less on average) and surrogate (from 1 to 3 milliseconds) relaxations, while the linear relaxation requires about 0.39 seconds on average.

\begin{table}
\centering
\scriptsize
\begin{tabular}{|c|c|c|c|}
Set &$n/m$ & $w_j$ in &Correlation \\
\hline
{\em SMALL}&$\{20/10, 40/10, 60/10, 20/20, 40/20, 60/20\}$&$[1-1000]$&\{uncorr., weekly, strongly\}\\
$FK_1$& $\{60/30, 45/15, 48/12, 75/15, 60/10, 100/10\}$&$[10-1000]$&\{uncorr., weakly, strongly, subset-sum\}\\
$FK_2 $&$\{120/60, 90/30, 96/24, 150/30, 120/20, 200/20\}$&$[10-1000]$&\{uncorr., weakly, strongly, subset-sum\}\\
$FK_3 $&$\{180/90, 135/45, 144/36, 225/45, 180/30, 300/30\}$&$[10-1000]$&\{uncorr., weakly, strongly, subset-sum\}\\
$FK_4$& $\{300/150, 225/75, 240/60, 375/75, 300/50, 500/50\}$&$[10-1000]$&\{uncorr., weakly, strongly, subset-sum\}\\
$Set_6$& \{150/80, 300/160, 600/350,1200/700, 2500/1400, &&\\
&5000/2800,10000/5600, 20000/13000, 45000/30000\} & \{[10-100], [10-1000], &\\
&&[10-10000]\} &\{uncorr., weakly, strongly\}\\
\hline
\end{tabular}
\caption{Description of the instances.} \label{tab:inst}
\end{table}


Tables  \ref{tab:Set6-1} and \ref{tab:Set6-2} report the computational results  on instances of $Set_6$. Observe that, on this set, the  ratio $n/m$ is very small, ranging from about 1.5  to 1.9. In each table, for each  $n$, $m$, $R$, and correlation class, the average over the 20 instances is reported.
In Table  \ref{tab:Set6-1}, the results on the smallest instances of  $Set_6$ are reported. More precisely, in the last two columns of the table, $g_{se-LP}$ is the percentage gap between $z_{seq}$ and  $z_{LP}$, and $g_{su-LP}$ is the percentage gap between $z_{surr}$ and  $z_{LP}$, computed as in Table \ref{tab:compare}, i.e.,
  $(z_{LP} - z_{seq})/z_{seq}\times 100$ and $(z_{LP} - z_{surr})/z_{surr}\times 100$, respectively. 
  First observe that, in Table  \ref{tab:Set6-1}, the linear relaxation requires about 15 seconds on average, but more than 130 seconds on instances with $m=1400$ and $n=2500$. On the other hand, the  computational times of the sequential and surrogate relaxations are negligible, about  3 ms and 5 ms on average, respectively. The sequential relaxation  always attains the best performances, with $g_{se-LP}=8\%$ and $g_{su-LP}=0.01\%$ on average.  The   sequential bounds are  smaller than  $z_{surr}$ especially on uncorrelated and weakly correlated instances. 
 
 Table  \ref{tab:Set6-2}  reports the results on the biggest instances of $Set_6$. On these instances, due to the high computational times to compute the linear relaxations, only the sequential and surrogate bounds are compared. Note that, $z_{seq}$ is always better than  $z_{surr}$. In fact,  $z_{seq}$ is about 20\%, 9\% and 2.5\% lower than  $z_{surr}$ on uncorrelated, weakly correlated and strongly correlated instances, respectively. Regarding the computational times for instances of $Set_6$, the computation of the surrogate and sequential bounds require about 0.036 and  0.08 seconds on average. However, the surrogate relaxation 
 is always faster than the  sequential relaxation  on uncorrelated and weakly  correlated instances. 

Summarizing, from a quality point of view, the sequential relaxation attains good performances  for uncorrelated and weakly correlated instances   with $n/m$ ratios smaller than or equal to 3. The surrogate relaxation produces the best results on strongly correlated instances and on instances with $n/m>3$.
On subset-sum instances with $n/m\ge 2$ the 
sequential, the surrogate and the linear relaxation attain the same results.
On instances with $n/m<2$  the sequential relaxation always produces the best results. Regarding the computational times, in general, the sequential upper bound can be computed with a small computational effort: it requires at most less than 0.3 seconds on the biggest instances of $Set_6$.
Such facts suggest that a combined use of sequential and surrogate relaxations could be effective when employed in enumeration solution schemes for MKP.


\begin{table}
\centering
\scriptsize
\begin{tabular}{|c|c|c|c||c|c|c|c|c|c|c|c|c|c|}\hline
$m$&$n$ &$n/m$& Corr. & $z_{seq}$ & $t_{seq}$ &  $z_{surr}$ & $t_{surr}$ &$z_{LP}$&$t_{LP}$&opt&$gap_{se}$&$gap_{su}$&$gap_{LP}$\\\hline
20&20&1&unc.&5963.31&$<$0.001&7659.5&$<$0.001&8071.08&0.57&4685.8&27.26&63.46&72.25\\
20&20&1&wea.&4633.687&$<$0.001&5696.6&$<$0.001&6158.33&0.42&2957&56.70&92.65&108.26\\
20&20&1&str.&9015.2&$<$0.001&9973.7&0.002&10758.88&0.38&6010.3&50&65.94&79.01\\
\hline10&20&2&unc.&7869.777&$<$0.001&7863&0.001&8075.55&0.54&7483.5&5.16&5.07&7.91\\
10&20&2&wea.&6001.897&$<$0.001&5986.3&$<$0.001&6162.66&0.58&5594.4&7.28&7.01&10.16\\
10&20&2&str.&10560.6&$<$0.001&10444.1&$<$0.001&10766.88&0.43&9740.9&8.42&7.22&10.53\\
\hline20&40&2&unc.&16078.99&$<$0.001&16368.9&$<$0.001&16484.34&0.57&15462.6&3.99&5.86&6.61\\
20&40&2&wea.&12178.64&$<$0.001&12351.4&$<$0.001&12454.56&0.42&11638.8&4.64&6.12&7.01\\
20&40&2&str.&21178.92&$<$0.001&20993.2&0.002&21248.26&0.37&20478.5&3.42&2.51&3.76\\
\hline20&60&3&unc.&24783.75&$<$0.001&24717.7&$<$0.001&24800.72&0.57&24632&0.62&0.35&0.68\\
20&60&3&wea.&18777.19&$<$0.001&18739.1&$<$0.001&18788.38&0.42&18661.4&0.62&0.42&0.68\\
20&60&3&str.&31800.32&$<$0.001&31608.3&0.002&31816.56&0.37&31535.6&0.84&0.23&0.89\\
\hline10&40&4&unc.&16464.61&$<$0.001&16395.4&$<$0.001&16488.41&0.57&16366.7&0.60&0.18&0.74\\
10&40&4&wea.&12446.41&$<$0.001&12397.3&$<$0.001&12459.13&0.48&12379.1&0.54&0.15&0.65\\
10&40&4&str.&21235.54&$<$0.001&21024.1&0.003&21256.26&0.42&21011.2&1.07&0.06&1.17\\
\hline10&60&6&unc.&24797.61&$<$0.001&24728.6&0.001&24804.74&0.57&24728.6&0.28&0&0.31\\
10&60&6&wea.&18790.28&$<$0.001&18746.1&$<$0.001&18793.29&0.42&18746.1&0.24&0&0.25\\
10&60&6&str.&31819.13&$<$0.001&31660.2&0.001&31825.56&0.41&31660.2&0.50&0&0.52\\
\hline Av & & & &16355.33&$<$0.001&16519.64&0.001&16734.09&0.47&15765.15&9.56&14.29&17.30\\
\hline 

\end{tabular}
\caption{Results on {\em SMALL} instances.} \label{tab:small}
\end{table}

\begin{table}
\centering
\scriptsize
\begin{tabular}{|c|c|c|c||c|c|c|c|c|c|c|c|c|c|}\hline
$m$&$n$ &$n/m$& Corr. & $z_{seq}$ & $t_{seq}$ &  $z_{surr}$ & $t_{surr}$ &$z_{LP}$&$t_{LP}$&opt&$gap_{se}$&$gap_{su}$&$gap_{LP}$\\\hline
10&60&6&unc.&23933.99&$<$0.001&23867.05&$<$0.001&23940.81&0.24&23867.05&0.28&0&0.31\\
10&60&6&wea.&16567.99&$<$0.001&16543.95&$<$0.001&16568.94&0.07&16540.9&0.16&0.02&0.17\\
10&60&6&str.&15076.14&$<$0.001&15071.55&0.001&15076.26&0.07&15071.55&0.03&0&0.03\\
10&60&6&s-s&14649.50&$<$0.001&14649.5&0.001&14649.50&0.07&14649.5&0&0&0\\
\hline 10&100&10&unc.&40256.88&$<$0.001&40207.05&0.001&40259.40&0.07&40207.05&0.12&0&0.13\\
10&100&10&wea.&27622.95&$<$0.001&27608.3&0.001&27623.77&0.07&27608.3&0.05&0&0.06\\
10&100&10&str.&25438.62&$<$0.001&25432.45&0.001&25438.69&0.07&25432.45&0.02&0&0.02\\
10&100&10&s-s&24729.45&$<$0.001&24729.45&0.001&24729.45&0.06&24729.45&0&0&0\\
\hline 12&48&4&unc.&18948.89&$<$0.001&18886.05&$<$0.001&18957.97&0.06&18871.35&0.41&0.08&0.46\\
12&48&4&wea.&13094.05&$<$0.001&13068.6&$<$0.001&13095.99&0.06&13024.1&0.54&0.34&0.55\\
12&48&4&str.&11961.72&$<$0.001&11956.2&0.001&11961.86&0.06&11955.5&0.05&0.01&0.05\\
12&48&4&s-s&11620.20&$<$0.001&11620.2&0.001&11620.20&0.06&11619.8&0&0&0\\
\hline 15&45&3&unc.&17751.83&$<$0.001&17787.75&$<$0.001&17857.79&0.06&17575.65&1.00&1.21&1.61\\
15&45&3&wea.&12802.02&$<$0.001&12832.45&$<$0.001&12860.73&0.06&12552.4&1.99&2.23&2.46\\
15&45&3&str.&12123.22&$<$0.001&12116.5&0.001&12123.56&0.06&12089.85&0.28&0.22&0.28\\
15&45&3&s-s&11811.25&$<$0.001&11811.25&0.001&11811.50&0.06&11790&0.18&0.18&0.18\\
\hline 15&75&5&unc.&30128.45&$<$0.001&30075.45&0.001&30133.00&0.06&30075.45&0.18&0&0.19\\
15&75&5&wea.&20677.82&$<$0.001&20660.85&0.001&20678.78&0.07&20649.85&0.14&0.05&0.14\\
15&75&5&str.&18804.37&$<$0.001&18798.95&0.001&18804.47&0.06&18798.95&0.03&0&0.03\\
15&75&5&s-s&18271.40&$<$0.001&18271.4&0.001&18271.40&0.06&18271.4&0&0&0\\
\hline 30&60&2&unc.&22413.89&$<$0.001&24725.45&0.001&24797.32&0.06&19412&15.46&27.37&27.74\\
30&60&2&wea.&16205.79&$<$0.001&17153.65&0.001&17179.30&0.07&12158.95&33.28&41.08&41.29\\
30&60&2&str.&15779.16&$<$0.001&15780.85&0.001&15786.51&0.07&12515&26.08&26.10&26.14\\
30&60&2&s-s&15368.90&$<$0.001&15368.9&0.001&15369.40&0.06&12152.5&26.47&26.47&26.47\\
\hline Av & & & &19001.60&$<$0.001&19125.99&0.001&19149.86&0.07&18400.79&4.45&5.22&5.35\\
\hline 

\end{tabular}
\caption{Results on $FK_1$ instances.} \label{tab:FK1}
\end{table}

\begin{table}
\centering
\scriptsize
\begin{tabular}{|c|c|c|c||c|c|c|c|c|c|c|c|c|c|}\hline
$m$&$n$ &$n/m$& Corr. & $z_{seq}$ & $t_{seq}$ &  $z_{surr}$ & $t_{surr}$ &$z_{LP}$&$t_{LP}$&opt&$gap_{se}$&$gap_{su}$&$gap_{LP}$\\\hline
20&120&6&unc.&48174.79&$<$0.001&48129.25&$<$0.001&48177.20&0.55&48129.25&0.09&0&0.10\\
20&120&6&wea.&32949.39&$<$0.001&32935.85&$<$0.001&32949.96&0.55&32935.85&0.04&0&0.04\\
20&120&6&str.&30861.78&$<$0.001&30856.75&0.001&30861.87&0.42&30856.75&0.02&0&0.02\\
20&120&6&s-s&30014.25&$<$0.001&30014.25&0.001&30014.25&0.40&30014.25&0&0&0\\
\hline 20&200&10&unc.&80121.46&0.001&80095.65&$<$0.001&80122.77&0.39&80095.65&0.03&0&0.03\\
20&200&10&wea.&55415.80&0.001&55408.05&0.001&55416.14&0.37&55408.05&0.01&0&0.01\\
20&200&10&str.&51581.81&0.001&51577.65&0.001&51581.85&0.34&51577.65&0.01&0&0.01\\
20&200&10&s-s&50170.85&$<$0.001&50170.85&0.001&50170.85&0.34&50170.85&0&0&0\\
\hline 24&96&4&unc.&38643.99&$<$0.001&38593.05&0.001&38646.53&0.33&38590&0.14&0.01&0.15\\
24&96&4&wea.&26438.73&$<$0.001&26423.3&0.001&26439.44&0.33&26390.5&0.18&0.12&0.19\\
24&96&4&str.&24382.23&$<$0.001&24378&0.001&24382.33&0.32&24378&0.02&0&0.02\\
24&96&4&s-s&23701.70&$<$0.001&23701.7&0.001&23701.70&0.31&23701.7&0&0&0\\
\hline 30&90&3&unc.&36020.33&$<$0.001&36211.85&0.001&36260.69&0.31&35804.5&0.60&1.14&1.27\\
30&90&3&wea.&25017.34&$<$0.001&25115.25&$<$0.001&25132.18&0.31&24699.95&1.28&1.68&1.75\\
30&90&3&str.&23231.50&$<$0.001&23225.6&0.001&23231.64&0.30&23222&0.04&0.02&0.04\\
30&90&3&s-s&22596.10&$<$0.001&22596.1&0.001&22596.15&0.31&-&-&-&-\\
\hline 30&150&5&unc.&60157.34&0.001&60119.3&0.001&60158.92&0.30&60119.3&0.06&0&0.07\\
30&150&5&wea.&41743.95&0.001&41733.95&0.001&41744.33&0.30&41733.3&0.03&0&0.03\\
30&150&5&str.&38687.99&0.001&38683.65&0.001&38688.04&0.30&38683.65&0.01&0&0.01\\
30&150&5&s-s&37629.65&$<$0.001&37629.65&0.001&37629.65&0.29&37629.65&0&0&0\\
\hline 60&120&2&unc.&43708.45&$<$0.001&48671.45&0.001&48710.32&0.29&37433.85&16.76&30.02&30.12\\
60&120&2&wea.&31814.82&$<$0.001&33941.2&0.001&33955.99&0.29&23446.75&35.69&44.76&44.82\\
60&120&2&str.&31081.82&$<$0.001&31096.05&0.001&31101.54&0.30&23700.15&31.15&31.21&31.23\\
60&120&2&s-s&30256.55&$<$0.001&30256.55&0.001&30256.55&0.30&22970.65&31.72&31.72&31.72\\
  \hline Av & & & &38100.11&$<$0.001&38398.5396&0.001&38413.79&0.34&37464.88&5.13&6.12&6.16\\
 \hline
\end{tabular}
\caption{Results on $FK_2$ instances.} \label{tab:FK2}
\end{table}

\begin{table}
\centering
\scriptsize
\begin{tabular}{|c|c|c|c||c|c|c|c|c|c|c|c|c|c|}\hline
$m$&$n$ &$n/m$& Corr. & $z_{seq}$ & $t_{seq}$ &  $z_{surr}$ & $t_{surr}$ &$z_{LP}$&$t_{LP}$&opt&$gap_{se}$&$gap_{su}$&$gap_{LP}$\\\hline
30&180&6&unc.&71941.09&0.001&71914.45&$<$0.001&71942.72&0.33&71914.45&0.04&0&0.04\\
30&180&6&wea.&49796.58&0.001&49786.4&$<$0.001&49796.85&0.33&-&-&-&-\\
30&180&6&str.&46498.21&0.001&46494&0.003&46498.32&0.33&46494&0.01&0&0.01\\
30&180&6&s-s&45228.55&$<$0.001&45228.55&0.003&45228.60&0.33&45228.55&0&0&0\\
\hline 30&300&10&unc.&120389.12&0.001&120370.45&0.003&120390.18&0.29&120370.45&0.02&0&0.02\\
30&300&10&wea.&82745.83&0.001&82740&0.002&82746.01&0.29&82740&0.01&0&0.01\\
30&300&10&str.&77481.90&0.001&77476.65&0.004&77481.93&0.29&77476.65&0.01&0&0.01\\
30&300&10&s-s&75366.75&0.001&75366.75&0.002&75366.75&0.29&75366.75&0&0&0\\
\hline 36&144&4&unc.&57575.14&0.001&57539.95&0.003&57577.51&0.33&57539.95&0.06&0&0.07\\
36&144&4&wea.&40025.04&0.001&40012.3&0.002&40025.50&0.33&-&-&-&-\\
36&144&4&str.&36915.62&0.001&36910.25&0.003&36915.68&0.33&36910.25&0.01&0&0.01\\
36&144&4&s-s&35898.25&$<$0.001&35898.25&0.003&35898.25&0.33&35898.25&0&0&0\\
\hline 45&135&3&unc.&54277.77&$<$0.001&54490.65&0.002&54533.76&0.32&54024.4&0.47&0.86&0.94\\
45&135&3&wea.&37481.28&0.001&37637.95&0.002&37650.88&0.32&37212.8&0.72&1.14&1.18\\
45&135&3&str.&34968.79&$<$0.001&34963.4&0.003&34968.86&0.32&-&-&-&-\\
45&135&3&s-s&34019.90&$<$0.001&34019.9&0.003&34019.90&0.32&-&-&-&-\\
\hline 45&225&5&unc.&90132.20&0.001&90107.55&0.003&90133.34&0.29&90107.55&0.03&0&0.03\\
45&225&5&wea.&62117.52&0.001&62110.35&0.003&62117.92&0.28&-&-&-&-\\
45&225&5&str.&58143.37&0.001&58138.5&0.004&58143.41&0.29&58138.5&0.01&0&0.01\\
45&225&5&s-s&56557.60&$<$0.001&56557.6&0.003&56557.60&0.29&56557.6&0&0&0\\
\hline 90&180&2&unc.&64184.24&0.001&72464.65&0.003&72498.32&0.41&55174.75&16.33&31.34&31.40\\
90&180&2&wea.&47594.69&0.001&50527.85&0.003&50539.64&0.41&34645&37.38&45.84&45.88\\
90&180&2&str.&47221.35&0.001&47295.2&0.003&47300.31&0.41&36306.3&30.06&30.27&30.28\\
90&180&2&s-s&45982.15&$<$0.001&46036.7&0.003&46036.70&0.41&35208.8&30.60&30.75&30.75\\
 \hline Av & & & &57189.29&0.001&57670.3458&0.003&57682.04&0.33&58279.74&6.09&7.38&7.40\\
 \hline
\end{tabular}
\caption{Results on $FK_3$ instances.} \label{tab:FK3}
\end{table}

\begin{table}
\centering
\scriptsize
\begin{tabular}{|c|c|c|c||c|c|c|c|c|c|c|c|c|c|}\hline
$m$&$n$ &$n/m$& Corr. & $z_{seq}$ & $t_{seq}$ &  $z_{surr}$ & $t_{surr}$ &$z_{LP}$&$t_{LP}$&opt&$gap_{se}$&$gap_{su}$&$gap_{LP}$\\
\hline
50&300&6&unc.&120225.26&0.001&120208.50&$<$0.001&120226.33&0.80&120208.50&0.01&0&0.01\\
50&300&6&wea.&82739.73&0.001&82733.60&0.001&82739.92&0.82&-&-&-&-\\
50&300&6&str.&77626.97&0.001&77621.50&0.004&77627.01&0.80&77621.50&0.01&0&0.01\\
50&300&6&s-s&75513.00&0.000&75513.00&0.002&75513.00&0.61&75513.00&0&0&0\\
\hline 50&500&10&unc.&201363.40&0.002&201349.45&0.003&201364.22&0.64&201349.45&0.01&0&0.01\\
50&500&10&wea.&138576.70&0.002&138572.40&0.003&138576.79&0.66&138572.40&0&0&0\\
50&500&10&str.&129921.07&0.002&129915.00&0.006&129921.10&0.69&129915.00&0&0&0\\
50&500&10&s-s&126402.00&0.001&126402.00&0.002&126402.00&0.57&126402.00&0&0&0\\
\hline 60&240&4&unc.&95969.46&0.001&95946.15&0.003&95970.42&0.40&-&-&-&-\\
60&240&4&wea.&66057.19&0.001&66049.95&0.003&66057.52&0.39&-&-&-&-\\
60&240&4&str.&61995.19&0.001&61991.20&0.003&61995.23&0.40&-&-&-&-\\
60&240&4&s-s&60307.25&0.000&60307.25&0.003&60307.25&0.38&60307.25&0&0&0\\
\hline 75&225&3&unc.&89875.07&0.001&90309.05&0.004&90333.29&0.49&-&-&-&-\\
75&225&3&wea.&62468.21&0.001&62848.35&0.003&62855.22&0.49&-&-&-&-\\
75&225&3&str.&58349.92&0.001&58345.10&0.004&58349.95&0.49&-&-&-&-\\
75&225&3&s-s&56766.20&0.000&56766.20&0.003&56766.20&0.48&-&-&-&-\\
\hline 75&375&5&unc.&150371.81&0.001&150353.20&0.003&150373.02&1.12&-&-&-&-\\
75&375&5&wea.&104389.88&0.001&104384.15&0.003&104390.04&0.79&-&-&-&-\\
75&375&5&str.&97111.32&0.001&97105.70&0.005&97111.35&1.01&97105.70&0.01&0&0.01\\
75&375&5&s-s&94470.20&0.001&94470.20&0.003&94470.20&0.54&94470.20&0&0&0\\
\hline 150&300&2&unc.&105646.90&0.001&120378.55&0.002&120401.98&1.37&89253.45&18.37&34.87&34.90\\
150&300&2&wea.&77750.28&0.001&82871.45&0.004&82878.32&1.13&56429.45&37.78&46.86&46.87\\
150&300&2&str.&78215.07&0.001&78283.25&0.005&78288.99&1.26&57565.95&35.87&35.99&36.00\\
150&300&2&s-s&76161.25&0.001&76182.85&0.003&76182.85&0.60&55772.45&36.56&36.60&36.60\\
 \hline Av & & & &95344.72&0.001&96204.50&0.003&96212.59&0.71&98606.16&9.19&11.02&11.03\\
 \hline
\end{tabular}
\caption{Results on $FK_4$ instances.} \label{tab:FK4}
\end{table}

\begin{table}
\centering
\scriptsize
\begin{tabular}{|c|c|c|c||c|c|}\hline
$m$&$n$ &$n/m$& Corr. &$g_{se-LP}$&$g_{su-LP}$\\\hline
 30&90&3&unc.&0.67&0.13\\
 45&135&3&unc.&0.47&0.08\\
 75&225&3&unc.&0.51&0.03\\
30&90&3&wea.&0.46&0.07\\
45&135&3&wea.&0.45&0.03\\
75&225&3&wea.&0.62&0.01\\
30&90&3&str.&0&0.03\\
45&135&3&str.&0&0.02\\
75&225&3&str.&0&0.01\\
30&90&3&s-s&0&0\\
45&135&3&s-s&0&0\\
75&225&3&s-s&0&0\\
 24&96&4&unc.&0.01&0.14\\
 36&144&4&unc.&0&0.07\\
 60&240&4&unc.&0&0.03\\
24&96&4&wea.&0&0.06\\
36&144&4&wea.&0&0.03\\
60&240&4&wea.&0&0.01\\
24&96&4&str.&0&0.02\\
36&144&4&str.&0&0.01\\
60&240&4&str.&0&0.01\\
24&96&4&s-s&0&0\\
36&144&4&s-s&0&0\\
60&240&4&s-s&0&0\\
 30&150&5&unc.&0&0.07\\
 45&225&5&unc.&0&0.03\\
 75&375&5&unc.&0&0.01\\
30&150&5&wea.&0&0.02\\
45&225&5&wea.&0&0.01\\
75&375&5&wea.&0&0.01\\
30&150&5&str.&0&0.01\\
45&225&5&str.&0&0.01\\
75&375&5&str.&0&0.01\\
30&150&5&s-s&0&0\\
45&225&5&s-s&0&0\\
75&375&5&s-s&0&0\\
\hline
\end{tabular}
\caption{Gap results for $FK_2$--$FK_4$ instances with ratios $n/m=3,4,5$.} \label{tab:compare}
\end{table}

\begin{table}
\centering
\scriptsize
\begin{tabular}{|c|c|c|c|c||c|c|c|c|c|c|c|c|}\hline
$m$&$n$ &$n/m$&$R$& Corr. & $z_{seq}$ & $t_{seq}$ &  $z_{surr}$ & $t_{surr}$ &$z_{LP}$&$t_{LP}$&$g_{se-LP}$&$g_{su-LP}$\\
\hline
80&150&1.88&100&unc.&5285.17&$<$ 0.001&6011.55&$<$ 0.001&6014.50&0.56&13.80&0.05\\
80&150&1.88&100&wea.&4220.98&$<$ 0.001&4492.4&$<$ 0.001&4493.52&0.57&6.46&0.02\\
80&150&1.88&100&str.&5097.86&$<$ 0.001&5113.15&0.001&5118.76&0.55&0.41&0.11\\
\hline 80&150&1.88&1000&unc.&53472.20&0.001&61305.85&0.001&61336.80&0.52&14.71&0.05\\
80&150&1.88&1000&wea.&39046.66&0.001&41810.45&0.001&41821.65&0.41&7.11&0.03\\
80&150&1.88&1000&str.&39177.02&0.001&39269.65&0.001&39274.12&0.42&0.25&0.01\\
\hline 80&150&1.88&10000&unc.&540793.37&0.001&616671.25&0.001&616992.50&0.42&14.09&0.05\\
80&150&1.88&10000&wea.&393215.31&0.001&421058.15&0.001&421183.73&0.40&7.11&0.03\\
80&150&1.88&10000&str.&373428.06&0.001&374758.15&0.006&374764.16&0.38&0.36&0\\
\hline 160&300&1.88&100&unc.&10554.14&0.001&12015.2&0.001&12017.24&1.45&13.86&0.02\\
160&300&1.88&100&wea.&8454.72&0.001&8991.25&0.001&8991.63&1.29&6.35&0\\
160&300&1.88&100&str.&10250.96&0.001&10287.5&0.001&10292.20&0.47&0.40&0.05\\
\hline 160&300&1.88&1000&unc.&104327.69&0.001&121290.7&0.001&121314.58&1.35&16.28&0.02\\
160&300&1.88&1000&wea.&79150.28&0.001&84439.7&0.001&84446.18&1.21&6.69&0.01\\
160&300&1.88&1000&str.&78727.26&0.001&78812.35&0.002&78817.35&1.34&0.11&0.01\\
\hline 160&300&1.88&10000&unc.&1058583.33&0.001&1218847.9&0.001&1219050.39&1.32&15.16&0.02\\
160&300&1.88&10000&wea.&774145.57&0.001&828211.75&0.001&828277.84&1.25&6.99&0.01\\
160&300&1.88&10000&str.&762157.28&0.001&762215.8&0.010&762220.62&1.34&0.01&0\\
\hline 350&600&1.71&100&unc.&20205.29&0.002&23895.3&0.001&23896.22&2.44&18.27&0\\
350&600&1.71&100&wea.&16684.80&0.002&17988.7&0.001&17988.96&1.04&7.82&0\\
350&600&1.71&100&str.&20340.81&0.001&20546.45&0.001&20551.09&0.95&1.03&0.02\\
\hline 350&600&1.71&1000&unc.&204882.70&0.002&243150.35&0.001&243163.82&3.02&18.68&0.01\\
350&600&1.71&1000&wea.&153559.08&0.002&166036.75&0.001&166040.90&2.37&8.13&0\\
350&600&1.71&1000&str.&153823.12&0.002&155697.55&0.004&155703.05&2.73&1.22&0\\
\hline 350&600&1.71&10000&unc.&2078107.61&0.002&2431168.2&0.001&2431307.17&3.10&17.00&0.01\\
350&600&1.71&10000&wea.&1549021.57&0.002&1671662.7&0.001&1671701.79&2.75&7.92&0\\
350&600&1.71&10000&str.&1492937.22&0.002&1498734.55&0.024&1498739.88&3.39&0.39&0\\
\hline 700&1200&1.71&100&unc.&39952.73&0.004&47787.55&0.002&47788.20&5.94&19.61&0\\
700&1200&1.71&100&wea.&33465.43&0.004&36088.2&0.002&36088.41&3.43&7.84&0\\
700&1200&1.71&100&str.&40638.11&0.002&41058.55&0.002&41062.78&3.28&1.05&0.01\\
\hline 700&1200&1.71&1000&unc.&412982.36&0.004&486229&0.001&486237.68&15.39&17.74&0\\
700&1200&1.71&1000&wea.&310003.60&0.005&332949.4&0.001&332951.46&8.44&7.40&0\\
700&1200&1.71&1000&str.&309706.65&0.004&311753.75&0.007&311758.68&4.89&0.66&0\\
\hline 700&1200&1.71&10000&unc.&4121581.48&0.005&4868058.4&$<$ 0.001&4868134.03&11.72&18.11&0\\
700&1200&1.71&10000&wea.&3082395.30&0.005&3306136&0.001&3306159.16&12.58&7.26&0\\
700&1200&1.71&10000&str.&2998861.93&0.004&3002797.35&0.044&3002802.23&13.68&0.13&0\\
\hline 1400&2500&1.79&100&unc.&86986.84&0.008&99836.7&$<$ 0.001&99837.26&28.54&14.77&0\\
1400&2500&1.79&100&wea.&70725.36&0.008&75211.85&$<$ 0.001&75211.85&27.75&6.34&0\\
1400&2500&1.79&100&str.&85257.79&0.004&85643.85&0.003&85648.42&30.96&0.46&0.01\\
\hline 1400&2500&1.79&1000&unc.&873657.01&0.010&1017653.7&0.001&1017658.21&97.50&16.48&0\\
1400&2500&1.79&1000&wea.&649399.51&0.010&696410.4&0.001&696411.44&47.01&7.24&0\\
1400&2500&1.79&1000&str.&644382.36&0.007&644953.6&0.012&644958.14&22.67&0.09&0\\
\hline 1400&2500&1.79&10000&unc.&8678526.05&0.010&10144737.8&$<$ 0.001&10144780.06&78.01&16.90&0\\
1400&2500&1.79&10000&wea.&6424109.63&0.010&6892621.15&0.002&6892633.80&130.81&7.29&0\\
1400&2500&1.79&10000&str.&6264418.04&0.009&6265005.15&0.090&6265009.82&105.11&0.01&0\\
\hline Av& & && &1003482.23&0.003&1095098.13&0.005&1095125.61&15.22&8.00&0.01\\
\hline
\end{tabular}
\caption{Results on  instances of $Set_6$ first part.} \label{tab:Set6-1}
\end{table}

\begin{table}
\centering
\scriptsize
\begin{tabular}{|c|c|c|c|c||c|c|c|c|}\hline

$m$&$n$ &$n/m$&$R$& Corr. & $z_{seq}$ & $t_{seq}$ &  $z_{surr}$ & $t_{surr}$ \\
\hline
2800&5000&1.79&100&unc.&174131.97&0.014&199976.5&$<$0.001\\
2800&5000&1.79&100&wea.&141397.62&0.012&150448.6&$<$0.001\\
2800&5000&1.79&100&str.&170593.25&0.008&171362.45&0.004\\
\hline 2800&5000&1.79&1000&unc.&1733779.56&0.019&2026512.55&0.001\\
2800&5000&1.79&1000&wea.&1295325.84&0.019&1389545.5&0.001\\
2800&5000&1.79&1000&str.&1295059.09&0.012&1296189.75&0.026\\
\hline 2800&5000&1.79&10000&unc.&17322644.64&0.023&20273817.3&0.001\\
2800&5000&1.79&10000&wea.&12852362.34&0.023&13799455.7&0.002\\
2800&5000&1.79&10000&str.&12524179.73&0.020&12525355.05&0.210\\
\hline 5600&10000&1.79&100&unc.&349151.19&0.030&400784.55&0.001\\
5600&10000&1.79&100&wea.&283165.76&0.024&301331.85&0.001\\
5600&10000&1.79&100&str.&341534.04&0.015&342973.9&0.008\\
\hline 5600&10000&1.79&1000&unc.&3469091.22&0.041&4054885.65&0.001\\
5600&10000&1.79&1000&wea.&2600634.72&0.040&2788691.8&0.001\\
5600&10000&1.79&1000&str.&2589502.98&0.024&2591776.15&0.051\\
\hline 5600&10000&1.79&10000&unc.&34714138.04&0.050&40650097.05&0.001\\
5600&10000&1.79&10000&wea.&25707763.86&0.051&27594278.75&0.002\\
5600&10000&1.79&10000&str.&25066973.99&0.038&25069333.35&0.298\\
\hline 13000&20000&1.54&100&unc.&669426.91&0.055&802070.55&0.001\\
13000&20000&1.54&100&wea.&561260.53&0.043&602619.65&0.001\\
13000&20000&1.54&100&str.&678745.71&0.031&686666.7&0.017\\
\hline 13000&20000&1.54&1000&unc.&6541607.46&0.086&8114919.4&0.002\\
13000&20000&1.54&1000&wea.&4999547.44&0.085&5567798.15&0.002\\
13000&20000&1.54&1000&str.&5021434.80&0.045&5191444.65&0.100\\
\hline 13000&20000&1.54&10000&unc.&65419254.77&0.104&81302573.65&0.002\\
13000&20000&1.54&10000&wea.&49493710.84&0.108&55209028.55&0.002\\
13000&20000&1.54&10000&str.&48547914.48&0.071&50198100.65&0.980\\
\hline 30000&45000&1.50&100&unc.&1488357.61&0.116&1804547&0.002\\
30000&45000&1.50&100&wea.&1253916.51&0.092&1355354.6&0.001\\
30000&45000&1.50&100&str.&1525743.16&0.071&1544526.25&0.029\\
\hline 30000&45000&1.50&1000&unc.&14193039.47&0.220&18258034.65&0.002\\
30000&45000&1.50&1000&wea.&10746949.94&0.203&12527337.45&0.002\\
30000&45000&1.50&1000&str.&10738679.17&0.105&11672608&0.294\\
\hline 30000&45000&1.50&10000&unc.&144512754.32&0.261&182724949.3&0.003\\
30000&45000&1.50&10000&wea.&108538518.40&0.266&124095787.9&0.003\\
30000&45000&1.50&10000&str.&105180360.39&0.155&112885427.5&3.551\\
\hline Av& & && &20076184.77&0.07&23060294.75&0.16\\
\hline
\end{tabular}
\caption{Results on  instances of $Set_6$ second part.} \label{tab:Set6-2}
\end{table}

\medskip\medskip
As an example, in order to asses whether the sequential bound can  be effectively employed in  the solution of multiple knapsack problems,  three new MKP instances have been considered, denoted as {\em Inst1}, {\em Inst2} and {\em Inst3}. {\em Inst1} is the MKP instance with $n=36$ and $m=30$ reported in Table \ref{tab:example2}.
For this instance the optimal solution value  $z_{MKP}=2000$, $z_{seq}=2033.31$  $z_{surr}=2103$ and $z_{LP}=2117.19$. 
 {\em Inst2} is generated by making three copies of each item and each knapsack of {\em Inst1}, and {\em Inst3} is the instance containing  6 copies of each item and each knapsack of {\em Inst1}. Hence, we have $n= 36\times 3= 108$ and $m=30\times 3=90$  in {\em Inst2}, and $n=36\times 6=216$ and $m=30\times 6=180$  in {\em Inst3}. 
 Furtheremore, we have $z_{MKP}=6000$, $z_{seq}=6099.92$  $z_{surr}=6350$ and $z_{LP}=6351.56$ in {\em Inst2}, and $z_{MKP}=12000$, $z_{seq}=12199.85$  $z_{surr}=12700$ and $z_{LP}=12703.12$ in {\em Inst3}. 
The three instances have been solved by Gurobi  both by the standard  formulation  \eqref{ilp:obj}--\eqref{ilp:bin} and on a {\em modified formulation}  obtained by simply adding to the standard formulation   the following valid cut
\begin{equation}\label{eq:sequb}
\sum\limits_{i=1}^{m}\sum\limits_{j=1}^{n} p_jx_{ij} \leq \lfloor z_{seq} \rfloor,\end{equation}
 where $z_{seq}$ is the  sequential bound obtained by Algorithm \ref{fig:algo}. 
 
In Table \ref{tab:res_example2},  the branch and bound nodes and the computational times required by Gurobi  for solving the three instances both by the formulation \eqref{ilp:obj}--\eqref{ilp:bin} and by the modified formulation \eqref{ilp:obj}--\eqref{ilp:bin}+ \eqref{eq:sequb} are reported. Note that,  the addition of Constraint \eqref{eq:sequb} allows a faster solution of  {\em Inst2} and {\em Inst3}, requiring in all the cases a  smaller number of branch and bound nodes.  

\begin{table}
\centering
\scriptsize
\begin{tabular}{|c||c|c|c|c|c|c|c|c|c|c|c|c|c|c|c|c|c|c|c|c|c|}
\hline w&33&35&37&47&64&30&35&36&39&39&40&41&33&35&37&47&64&33&35&37&47\\
&64&30&35&36&39&39&40&41&33&35&37&47&64&47&64&&&&&&\\
\hline
p&99&70&74&47&64&50&50&39&39&39&38&37&99&70&74&47&64&99&70&74&47\\
&64&50&50&39&39&39&38&37&99&70&74&47&64&100&50&&&&&&\\
\hline
c &47&64&40&64&47&64&40&64&47&64&40&64&40&64&40&64&40&64&47&64&40\\
&39&39&37&39&39&37&39&39&37&&&&&&&&&&&&\\
\hline
\end{tabular}
\caption{ The MKP instance {\em Inst1}.} \label{tab:example2}
\end{table}

\begin{table}
\centering
\scriptsize
\begin{tabular}{|c||cc|cc|}
\hline Instance&\multicolumn{2}{c|}{PLI \eqref{ilp:obj}--\eqref{ilp:bin}}&\multicolumn{2}{|c|}{PLI \eqref{ilp:obj}--\eqref{ilp:bin}+\eqref{eq:seq}}\\
\hline &\# BB nodes&time (sec.)&\# BB nodes&time (sec.)\\
\hline {\em Inst1}& 511&1.48 &9 & 1.53 \\
\hline {\em Inst2}& 1082&4.64 &39 & 3.25 \\
\hline {\em Inst3}& 1784.0&26.98 &31 & 10.61 \\
\hline
\end{tabular}
\caption{ Gurobi results.} \label{tab:res_example2}
\end{table}

\section{Conclusions}
In this paper, a new technique for computing upper bounds for MKP is proposed, based on the idea of relaxing MKP to a Bounded Sequential  Multiple Knapsack Problem. The sequential upper bound turns out to be not worse than the linear relaxation of the standard formulation. Computational results on benchmark instances from the literature shows that the sequential upper bound can be computed by small a computational effort, and outperforms the bound produced by the surrogate relaxation  when the ratio $n/m$ is smaller than or equal to 3 and weights and profits are uncorrelated or weakly correlated. On the other hand, for bigger $n/m$ ratios or strongly correlated instances, the surrogate bound  is better than the sequential bound. Future research includes: $(i)$ the designing of exact solution schemes for MKP embedding the sequential  relaxation; $(ii)$ investigating whether the sequential  relaxation can be applied to other optimization problems.

\end{document}